\documentclass{amsart}
\usepackage[utf8]{inputenc}
\usepackage{hyperref}
\usepackage{cleveref}
\usepackage{thmtools}
\usepackage{thm-restate}
\usepackage{cleveref}

\inputencoding{latin2}

\title{Rainbow paths and large rainbow matchings}
\author{Ron Aharoni} \thanks{We acknowledge the financial support from the Ministry of Educational and Science of the Russian Federation in the framework of MegaGrant no. 075-15-2019-1926 when the first author worked on Section 3 of the paper.\\
\indent Ron Aharoni:
Department of Mathematics, Technion, Israel and MIPT.
\url{ra@technion.ac.il}.
The research of R.~Aharoni was supported in part by the Israel Science Foundation (ISF) grant no. 2023464 and the Discount Bank Chair at the Technion. This paper is part of a project that has received funding from the European Union's Horizon 2020 research and innovation programme under the Marie Skldowska-Curie grant agreement no.\ 823748.
}
\author{Eli Berger} \thanks{
Eli Berger: Department of Mathematics, University of Haifa,  Israel. \url{berger@math.haifa.ac.il}.}

\author{Maria Chudnovsky} \thanks{Maria Chudnovsky: Department of Mathematics, Princeton University, USA, \url{mchudnov@math.princeton.edu}. Supported by  NSF Grant DMS-1763817.}
\author{Shira Zerbib} \thanks{Shira Zerbib: Department of Mathematics, Iowa State University, USA.  \url{zerbib@iastate.edu}. Supported by NSF grant DMS-1953929. \\ \indent  The authors were supported by US-Israel Binational Science Foundation (BSF) grant no. 2016077.}

\usepackage{amsthm,amssymb,amsmath,enumerate,graphicx, tikz}
\usepackage{amscd}
\usepackage{setspace}
\usepackage{comment}
\usepackage{hyperref}

\newtheorem{theorem}{Theorem}[section]

\newtheorem*{theorem*}{Theorem}

\newtheorem{lemma}[theorem]{Lemma}

\newtheorem{question}[theorem]{Question}

\newtheorem{claim}[theorem]{Claim}
\newtheorem{conjecture}[theorem]{Conjecture}
\theoremstyle{definition}
\newtheorem{definition}[theorem]{Definition}

\theoremstyle{remark}
\newtheorem{remark}[theorem]{Remark}

\newcommand{\cf}{\mathcal{F}}

\newcommand{\F}{\mathcal{F}}
\newcommand{\G}{\mathcal{G}}
\newcommand{\cs}{\mathcal{S}}

\newcommand{\Z}{\mathcal{Z}}

\newcommand{\ch}{\mathcal{H}}
\newcommand{\cg}{\mathcal{G}}

\newcommand{\T}{\mathcal{T}}

\begin{document}

\begin{abstract}

A conjecture of the first two authors is that   $n$ matchings of size $n$
in any graph have a rainbow matching of size $n-1$.  We prove a lower bound of $\frac{2}{3}n-1$, improving on the trivial $\frac{1}{2}n$, and an analogous result for hypergraphs. 
For 
$\{C_3,C_5\}$-free graphs and for disjoint matchings we obtain a lower bound of $\frac{3n}{4}-O(1)$. We also discuss a conjecture on rainbow alternating paths, that if true would yield a lower bound of $n-\sqrt{2n}$. We prove the non-alternating (ordinary paths) version of this conjecture. 

\end{abstract}

\maketitle

\section{Introduction}

Using a common terminology, a  {\em family} is a multiset, namely elements may  repeat. While the notation for sets uses curly brackets, families are denoted using ordinary parentheses, so in the notation $(a_i \mid i \in I)$  equality may hold between some $a_i$s. 

Given a family $\cs=(S_1, \ldots,S_n)$ of sets,
we write $\lVert \cs \rVert$ for $\sum_{i \le n}|S_i|$.

For such $\cs$, an {\em $\cs$-rainbow set} (or just ``rainbow set'' if the identity of $\cs $ is  clear from the context) is the range of a partial choice function. That is, it is a set $\{x_{i_j} \mid 1\le i_1<i_2<\ldots <i_m\le n\}$, where $x_{i_j} \in S_{i_j} ~~(1\le j \le m)$.

For integers $a,b,c,r$ we write $(a,b)\to_r c$ if every family of $a$ matchings in an $r$-uniform hypergraph, each of size $b$, has a rainbow matching of size  $c$. If we demand this condition only for $r$-partite hypergraphs, we write $(a,b)\to_r^P c$. If $r=2$ we omit its mention and write just $(a,b)\to c$ and $(a,b)\to^P c$, respectively.

A famous conjecture of Ryser-Brualdi-Stein \cite{brualdiryser, stein} is that in an $n \times n$ Latin square there exists a transversal (sub-permutation submatrix with distinct symbols) of size $n-1$, and that for $n$ odd there exists a transversal of size $n$. In \cite{ab} the first part of this conjecture was strengthened to:

\begin{conjecture}[Aharoni-Berger \cite{ab}]\label{ab}
$(n,n)\to n-1. $
\end{conjecture}

In \cite{ab} this was conjectured only for bipartite graphs, but there is no counterexample known also for general graphs.

Here are some known facts:

\begin{enumerate}[ (F1)]

    \item $(n,\frac{3}{2}n) \to^P n$   \cite{akz32}.
    \item $(n,\frac{3}{2}n+o(n)) \to n$ \cite{clemens}.
    \item $(2n-1,n)\to^P n$ (The Latin rectangle case was proved in \cite{drisko}, the general case in \cite{ab}).

    \item $2n-1$ matchings in a bipartite graph, of respective sizes $1, 2, 3, \ldots ,n-1,n,n,\ldots ,n$
($n$ is repeated $n$ times) have a rainbow matching of size $n$.  \cite{akzunique}.    

\item \label{strongerdrisko} Let $m \ge n$. Any $2n-1$ matchings of size $m$ have a matching of size $m$ representing $n$ of them \cite{ab}.

     \item \label{woolbrightfact} $(n,n) \to^P n-\sqrt{n}$ \cite{woolbright}.
     \item $(\lfloor \frac{k+2}{k+1}n\rfloor -(k+1),n) \to^P n-k$ \cite{BGS}.
    
    \item $(3n-3,n)\to n$ \cite{abkk}.
    \item 
    $2n$ matchings of size $n$ have a rainbow set with a fractional matching of size $n$ \cite{ahj}.
    
    \item $(n,n+o(n)) \to^P n$ \cite{pok}.
    \item $n$ disjoint matchings of size $n+o(n)$ in any graph
    have a rainbow matching of size $n$. \cite{GRWW}.
\end{enumerate}

It was conjectured in \cite{BGS} and in  \cite{ab} that $(2n,n) \to n$. This would follow, in particular, from Conjecture \ref{ab}.

In $r$-uniform hypergraphs a greedy argument  yields 
$(n,n)\to_r \lceil \frac{n}{r}\rceil$. We shall improve this to
$(n,n)\to_r \frac{n}{r-\frac{1}{2}}+O(\frac{1}{r})$.  In the case $r=2$, the explicit calculation yields $(n,n) \to \frac{2}{3}n-\frac{4}{3}.$ This improves upon the result $(n,n) \to \frac{2}{3}n-o(n)$, that follows from  (F2). When the graph is $\{C_3,C_5\}$-free or  the given matchings are disjoint,  the bound can be improved to $\frac{3}{4}n-O(1)$: 
\begin{theorem}\label{main}\hfill
\begin{enumerate}
    \item 
 $(n,n) \to \frac{2}{3}n-\frac{4}{3}.$
 \item $(n,n)\to_r \frac{n-\frac{1}{2}-\frac{3}{4r-6}}{r-\frac{1}{2}}.$
   \item 
If  $\F=(F_1, \ldots ,F_n)$ is a family of matchings of size $n$ in a $\{C_3,C_5\}$-free graph, then there exists a rainbow matching of size at least $ \frac{3}{4}n-\frac{9}{4}$.  
 
\item If  $\F=(F_1, \ldots ,F_n)$ is a family of disjoint matchings of size $n$ in a graph, then there exists a rainbow matching of size at least $ \frac{3}{4}n-\frac{9}{2}$.
\end{enumerate}
\end{theorem}

Theorem \ref{main} is proved in Section \ref{sec:main}. 
\medskip

A main tool used in the study of rainbow matchings is alternating paths. For example, among facts (F1-F10),  only for two, (F4) and (F8), there is no alternating paths proof known (the existing proofs use topology).  Often the proof goes through results on rainbow directed paths. 
To state the latter, we  need some definitions.

Though the graphs we are considering are all undirected, all paths will be assumed below to be directed. The initial and terminal vertices of a path $P$ are denoted by $in(P)$ and $ter(P)$, respectively.
\begin{definition}
Given two sets $S,T$ of vertices, a directed path $P$ is called an $S-T$-{\em path} if $in(P) \in S,~ ter(P) \in T$, and $V(P)\cap 
(S\cup T)=\{in(P), ter(P)\}$.\end{definition}

\begin{definition}\hfill
\begin{enumerate}
    \item 
Let $\ch$ be a family of (not necessarily distinct or disjoint) sets of paths.  A path is called {\em strongly rainbow} if each of its edges is chosen from a path belonging to a  different $F\in \ch$.
\item Let $M$ be a matching, and let $\ch$ be a family of sets of $M$-alternating paths.  An $M$-alternating  path is called {\em strongly rainbow} if each of its non-$M$ edges is chosen from a path belonging to a  different $F\in \ch$.

This is ``double rainbow-ness": a rainbow set of edges, one from each path in a rainbow 
set of paths. 

\item
Let $\F=(P_1, \ldots ,P_m) $ be an ordered  family of directed paths. A directed path $P$ is called {\em rainbow-monotone} (with respect to $\F$) if its edges, as ordered by $P$, are $(e_1, \ldots ,e_m)$, satisfy $e_i \in E(P_{j_i})$, where
$j_1<j_2<\ldots <j_m$.

\end{enumerate}
\end{definition}

Similar definitions applys to $M$-alternating paths, where $M$ is a matching, and the edges referred to in the definitions are the non-$M$ edges.

Here is a list of known facts about rainbow paths. Part (\ref{alternatingsingleton}) was proved in \cite{ABCHS}, Part (\ref{stdisjoint}) was proved in 
\cite{akz32}, and the others are taken from  \cite{akzunique}.   
We use   the following notation:  $G$ is an undirected graph, $M$ is a matching in $G$,  $S,T$ are subsets of $V(G)$, and
$Y=V(G)\setminus (S\cup T)$.
In each of these facts we add something fact that was in fact included in the original proof, without explicit mention - rainbow monotonicity.  

\begin{theorem}\label{combinations} \hfill
\begin{enumerate}
    \item 
    \label{disjointsimple}
       Let $F$ be a sequence of (not necessarily distinct) directed $S-T$  paths. If $S\cap T=\emptyset$ and $|F|>|Y|$   then there exists a directed $S-T$ rainbow-monotone path. 
    
    \item {\em (Corollary of (\ref{disjointsimple}))}: Let  $G$ be bipartite. Let $F$ be sequence of augmenting $M$-alternating paths with $|F|>|M|$. Then there exists a rainbow-monotone augmenting  $M$-alternating path. 
    
%\item Suppose $G$ is directed. If $F$ is a family of $S-T$ directed paths and  $|F|\ge 2|Y|$, then there exists a rainbow directed $S-T$ path. 
    
    \item \label{alternatingsingleton} 
In a general graph, if $F$ is a family of augmenting $M$-alternating paths with $|F|>2|M|$ then there exists a rainbow augmenting $M$-alternating path. (No monotonicity claim in this case).

     \item \label{stdisjoint}{\em (Strengthening of (\ref{disjointsimple})):}  
     Let $\ch$ be a family of sets of disjoint directed $S-T$ paths. If  $S \cap T=\emptyset$ and $\lVert \ch \rVert > |Y|$ then there exists a directed strongly rainbow $S-T$ path.  (Recall, $\lVert \ch \rVert = \sum_{H\in \ch} |H|$.)
     
     The proof in \cite{akzunique} yields a monotonicity, stronger version: 
     
     \item \label{monotone_disjoint}
     
     For $S,T$ be as above, 
     any sequence of more than $|Y|$ $S-T$ paths has a rainbow-monotone $S-T$ path.
     \item \label{bipartitefamiliesalt} {\em (Corollary of 
     (\ref{stdisjoint})):} If $G$ is bipartite, $M$ is a matching, $\ch$ is a family of sets of sets of disjoint augmenting $M$-alternating paths and $\lVert \ch \rVert > |M|$, then there exists a strongly rainbow augmenting alternating path. Again, this is a corollary of a  monotonicity version.
    \end{enumerate}
\end{theorem}

%\textcolor{red}{Where are the proofs of these statements? The structure of this section is not clear to me.}

All these are sharp. For example, the sharpness of (\ref{stdisjoint}) is  shown by the family $\ch$ consisting of the path $sv_1v_2\ldots v_mt$ repeated $m$ times (here $S=\{s\}, T=\{t\}$).

%The need in the strengthening to families of sets of paths (as in Theorem  \ref{combinations}(4) and Conjecture \ref{conjpaths}), originates in the following. If $M,F$ are matchings, then $M\cup F$ contains a family $A(M,F)$ of size 
%$|M|-|F|$ of disjoint augmenting $F$-alternating paths. Given a rainbow matching $F$, and matchings $M_1, \ldots ,M_k$ not represented in $F$, constructing an augmenting rainbow alternating path of the families $A(M_i,F)$ can be used to enlarge  $F$, keeping rainbow-ness.

%\begin{definition}[pseudo-forest]
%A collection of trees is said to form a {\em pseudo-forest}  if, whenever two trees share a vertex, this vertex is the root of at least one of them.
%\end{definition}

 The proof of fact (F5) 
 uses strongly rainbow alternating paths.  Extending that proof to  general graphs will require proving the following:

 \begin{conjecture}\label{conjpaths}
 Let $G$ be a graph and let $M$ be a matching in $G$. Let  $\ch$ be a family of
 sets of disjoint augmenting $M$-alternating paths in $G$. If $\lVert \ch \rVert >2|M|$ then there exists a strongly rainbow augmenting $M$-alternating path.% (namely a path composed of a choice of at most one edge from some path in each set). 
 \end{conjecture}
 
 Let us show how this conjecture implies 
 $(n,n)\to  n-\sqrt{2n}$.
 Let $F_1, \ldots ,F_n$ be matchings of size $n$ in a general graph. Let 
 $M$
 be a maximal size  rainbow matching. We claim that $|M|\ge  n-\sqrt{2n}$. Assume for contradiction that $|M| <  n-\sqrt{2n}$. For every $i \in I$ we have $|F_i| - |F| >\sqrt{2n}$, implying that $M\triangle F_i$ contains a set $H_i$  of $\sqrt{2n}$ $M$-augmenting alternating paths. Let $\ch=(H_i, ~i \in I)$.  Then $\lVert \ch \rVert > 2n >2|M|$, and by the conjecture there exists a strongly rainbow $M$-augmenting path $P$. Then $M \triangle P$ is a strongly rainbow matching larger than $|M|$, a contradiction.  
 \medskip
 
Our second main theorem is the ordinary (non-alternating) path version of Conjecture \ref{conjpaths}
   -- a possible first step.

\begin{theorem}\label{symmetric}
 Let the vertex set of an undirected graph be partitioned into two sets, $S,Y$. Let $\ch$ be a family of sets of paths, each consisting of disjoint  $S-S$ paths.  If $||\ch||>2|Y|$ then there exists a strongly rainbow $S-S$ path.
\end{theorem}
This is sharp, as shown by the following construction. Let $S=\{u,v\}$, and let $Y=\{y_1,\ldots, y_m\}$. Let $\ch$ 
consist of $2m$ families $\F_1, \ldots ,\F_{2m}$, each $\F_i$ consisting of a  single path $F_i$, where
$F_i=uy_1y_2\ldots ,y_{m}v$ for $i \le m$, and    $F_i=vy_{m}y_{m-1}\ldots ,y_1v$ for $m<i\le 2m$. 
%A proof of this theorem when the $H_i$s are each a singleton path is at the core of the proof in \cite{ABCHS}.

We shall prove a monotonicity version: 

\begin{theorem}\label{monotone_symmetric}
 Let  $S,Y$ be as above, and let $\F$ be a sequence of $2|Y|+1$ $S-S$ directed paths.  Then there exists a  rainbow-monotone $S-S$ directed path.
\end{theorem}

Theorem \ref{monotone_symmetric} implies Theorem \ref{symmetric}. Given a system of sets of disjoint directed paths $\ch=(H_1, \ldots, H_m)$, order $\bigcup \ch$ so that all paths in $H_i$ appear before all paths in $H_j$ whenever $i<j$. A rainbow-monotone path is then clearly strongly rainbow.  

In Section 3 we shall give two proofs for Theorem \ref{monotone_symmetric}. They are quite different, and may point at two possible  proof strategies in the conjectured alternating paths case.

\begin{remark}
 Since the  submission of this paper, the following beautiful results have been proved by Correia,  Pokrovskiy and  Sudakov \cite{cps}:

 \begin{itemize}
     \item[(a)]

$(n,n) \to n-o(n)$,
 
     \item[(b)] $(n,n+o(n)) \to n$.
  \end{itemize}

 Part (b) is easily seen to imply part (a). Interestingly, the authors  prove (a), and give a general probabilistic construction  showing that (a) implies (b). 
\end{remark}

\section{Proof of Theorem \ref{main}}\label{sec:main}

\subsection{Definitions and lemmas}

\begin{definition}
 Two sets of edges, $A,B$  are {\em orthogonal} if  $|a \cap b|=1$ for every $a \in A, b \in B$. 
\end{definition}

%\begin{lemma}\label{mols}
%There are no $r+2$ mutually cross-intersecting matchings of size $r$ in an $r$-uniform hypergraph. 
%\end{lemma}
%    \begin{proof}
%Suppose to the contrary that  $H_1, \ldots ,H_{r+2}$ are cross-intersecting matchings of size $r$ in an $r$-uniform hypergraph. 

%We first claim that for every two edges $e\in H_i, ~f \in H_j$ with $i\neq j$ we have $|e\cap f|=1$. Indeed, since $H_i, H_j$ are cross-intersecting, $|e\cap f|\ge 1$. Suppose $|e\cap f|\ge 2$. Then the $r-1$ edges in $H_j\setminus \{f\}$ must intersect $e\setminus f$, but since $|e\setminus f|\le r-2$, at least two of the edges  in $H_j\setminus \{f\}$ share a vertex in $e$, contradicting the fact that $H_j$ is a matching.

%Let $h,h'$ be two edges in $H_{r+2}$, and let $u\in h'$ be a vertex. Let $1\le i \le r+1$. By the assumptions that $H_i$ is a matching of size $r$ and that $H_i$ and $H_{r+2}$ are cross-intersecting, we have that  $H_i$ contains an edge $h_i$ with $u \in h_i$. Now, each of the edges $h_1,\dots, h_{r+1}$ intersects $h$ at some  vertex $p_i \in h$ and, by the claim above, $p_1,\dots,p_{r+1}$ are distinct vertices. But this contradicts the fact that $|h|=r$.
 %   \end{proof}

The following lemma  is (up to niceties) a well-known fact,  usually expressed as ``there do not exist $n$ mutually orthogonal Latin squares of order $n$". 

\begin{lemma}\label{mols}
Let $M_0, M_1, \ldots ,M_t$ be mutually orthogonal matchings   in an $r$-uniform hypergraph, where $M_0=\{h\}$ and $|M_i|=r$ for $1\le i \le t$. Then $t \le r$.
\end{lemma}
 \begin{proof}
%Suppose to the contrary that  $H_1, \ldots ,H_{r+1}$ are cross-intersecting matchings of size $r$ in an $r$-uniform hypergraph, all orthogonal to a given edge $h$. 

%We first claim that for every two edges $e\in H_i, ~f \in H_j$ with $i\neq j$ we have $|e\cap f|=1$. Indeed, since $H_i, H_j$ are cross-intersecting, $|e\cap f|\ge 1$. Suppose $|e\cap f|\ge 2$. Then the $r-1$ edges in $H_j\setminus \{f\}$ must intersect $e\setminus f$, but since $|e\setminus f|\le r-2$, at least two of the edges  in $H_j\setminus \{f\}$ share a vertex in $e$, contradicting the fact that $H_j$ is a matching.

 Let $e\in M_{1}$ and let $u\in e \setminus h$. Since the edges in  $M_i$ meet $e$ at distinct vertices, each 
  $M_i, ~2\le i \le t,$  contains an edge $e_i$ meeting $h$ at $u$. Orthogonality implies that the edges $e_i, ~2\le i \le t,$ meet $h$ at distinct vertices in $h\setminus e$, and hence their number $t-1$ is at most $|h\setminus e|=r-1$.
    \end{proof}

 The proofs of all four parts of Theorem \ref{main} start the same way. 
Let  $\F=(F_1, \ldots ,F_n)$ be a collection of matchings of size $n$ in an $r$-uniform hypergraph (in the last three parts of the theorem $r=2$). 
Let $R$ be a rainbow matching of maximal size, say $q$.
%and let $\alpha=\frac{q}{n}-1. $ 
Let $\G$ be the collection  of matchings in $\F$ not represented in $R$. 
%Let $F_i$ be a matching belonging to $\G$.

Given $G \in \G$ and $e \in R$ we say that $e$ is $G$-{\em wasteful} if either (1) $e$ meets at most $r-1$ edges from $G$ (meaning $e$ is not using its full hitting potential with respect to $G$), or (2) there exists an edge $g \in G$ meeting both $e$ and another edge $f \in R$ (meaning $e$ is not essential for hitting $g$). 
For an edge $e \in R$  let $T(e)$ be the set of matchings $G \in \G$ for which $e$ is  non-$G$-wasteful. By the definition of ``wastefulness'' the following holds:

\vspace{0.3cm}

%\begin{indent}
$(*)~~ G \in T(e)$  if and only if there exist $r$ edges in $G$ intersecting $ e$ and not intersecting any other edge of $R$.
    %\end{indent}
\vspace{0.3cm}

 For a matching $G\in \G$ let 
$T^G$ be the set of edges $e \in R$ such that $G \in T(e)$. Write $t(G)=|T^G|$.

\begin{lemma}\label{punchline}
$|T(e)| \le r$ for every $e \in R$.
\end{lemma}
\begin{proof}
 Suppose  there are $r+1$ matchings in $T(e)$. For each $G \in T(e)$ let $G_e$ be the set of edges of $G$ meeting $e$. Then $G_e$ is a matching of size $r$ orthogonal to $e$. By Lemma \ref{mols}
 there exist $A, B \in T(e)$ such 
 that $A_e, B_e$ are  not  cross-intersecting, meaning that there are $a \in A_e$ and $b \in B_e$, such that $a \cap b =\emptyset$. Since $A, B \in T(e)$, $a$ and $b$ do not intersect any edge in $R\setminus \{e\}$. Then  $R\cup \{a,b\}\setminus \{e\}$ is a rainbow matching, and it is larger than $R$, contradicting the maximality of $R$. 
\end{proof}
 Lemma \ref{punchline} implies 
\begin{equation}\label{eq1}
\sum_{e \in R}|T(e)| \le rq.
\end{equation}

%The logic is now that if $q$ is small, then $T^G$ is large for every $G \in \G$ - for example, if $q=\frac{n}{r}$ then $T^G=R$ for every $G \in \G$. 
%This will yield a violation of \eqref{eq1}.

\subsection{Parts (1) and (2) of Theorem \ref{main}}
Given $t=t(G)$, the maximal size of $G$ is attained when the edges in   $R \setminus T^G $ are arranged in pairs, so that the two edges in each pair meet at most $2r-1$ distinct edges of $G$, and the last edge, if it exists (namely if $R$ is odd), meets at most $r-1$  additional  edges of $G$.

 Since $\bigcup R $ is a cover for $G$, this implies
$$n=|G|  \le rt +   (q-t)(r-\frac{1}{2})= qr-\frac{1}{2}(q-t).$$
%(A sanity check - when $q=\frac{n}{r}$ we have equality, since then $t=q$.)

So, $t \ge 2n-2qr+q$. 
Summing up over all $G \in \G$ and changing the order of summation, we get 
$$ \sum_{e \in R}|T(e)| =\sum_{G \in \G} |T^G| \ge  (2n-q(2r-1))|\G|=(2n-q(2r-1))(n-q). $$

Combining this with (\ref{eq1}) yields
\begin{equation}\label{eqq}
   (2n-q(2r-1))(n-q) \le rq.
\end{equation}
 
\begin{claim}\label{claimq}
 $q> \frac{n-\frac{1}{2}-\frac{3}{4r-6}}{r-\frac{1}{2}}$.
\end{claim}
\begin{proof}
Rearranging (\ref{eqq}), we get $$(2r-1)q^2 -(n+2rn+r)q + 2n^2\le 0.$$
Solving for $q$, we have $q\ge q_1$, where 
\begin{equation}\label{square}
q_1 = \frac{n+2rn+r-\sqrt{(n+2rn+r)^2-8(2r-1)n^2}}{2(2r-1)}.
\end{equation}

Let 
\begin{equation*}
    \begin{split}\Delta &= (n+2rn+r)^2-8(2r-1)n^2 \\&=  n^2(4r^2-12r+9)+n(4r^2+2r)+r^2
     \end{split}
\end{equation*}
    be the discriminant in (\ref{square}). Completing $\Delta$ to a square entails $$\Delta < n^2(2r-3)^2+2n(2r^2+r)+\Big(\frac{2r^2+r}{2r-3}\Big)^2,$$ implying that $$\sqrt{\Delta}<  n(2r-3) + \frac{2r^2+r}{2r-3}=n(2r-3)+r+2 +\frac{6}{2r-3}.$$ 

Thus 
\begin{equation*}
    \begin{split}
    q_1 &= \frac{n+2rn+r-\sqrt{\Delta}}{2(2r-1)} \\ &>
    \frac{n+2rn+r-n(2r-3)-r-2 -\frac{6}{2r-3}}{2(2r-1)} \\ &=\frac{n-\frac{1}{2}-\frac{3}{4r-6}}{r-\frac{1}{2}}.    
    \end{split}
\end{equation*}
This shows  
\begin{equation}\label{eq4}
    q> \frac{n-\frac{1}{2}-\frac{3}{4r-6}}{r-\frac{1}{2}},
\end{equation}

\end{proof}

The claim  entails Part  (2) of Theorem \ref{main}. 
Part (1) follows upon plugging in  $r=2$.

%Both proofs use a  double counting argument, similar to the one used in the previous section, but instead of counting only the edges participating in  augmenting paths of size 3 (``non-wasteful edges"), we  count also the edges participating in augmenting paths of size 5 (``half-wasteful edges", defined below).
\medskip

\subsection{More definitions}
For the proofs of parts (3) and (4) of Theorem \ref{main} we  need additional definitions. 
For a matching $G\in \G$, we say that a pair of edges  $\{e,f\}$ in $R$ is  {\em half-$G$-wasteful} if there are three edges $g_e,g_f,g_{ef}\in G$ satisfying the  following: 
\begin{itemize}
    \item 
    $g_e$ intersects $e$ and no other edge in $R$, 
    \item
    $g_f$ intersects $f$ and no other edge in $R$, and
    \item $g_{ef}$ intersects both $e$ and $f$.
\end{itemize}
 We also say then that each of $e$ and $f$ are {\em half-$G$-wasteful}.
 Denote by $HW(e)$  the set of all $G \in \G$ for which the edge $e\in R$ is  half-$G$-wasteful. Let $W^G$ be the set of edges $e \in R$ such that $G \in HW(e)$, and write $w=|W^G|$.

Let $B$ be the bipartite graph with respective sides  $\G$ and $R$, in which   $G\in \G$ is adjacent to  $e\in R$ if and only if $e\in HW(G)$. Then 
\begin{equation}\label{eqB}
   \sum_{e \in R}|HW(e)|= \sum_{e \in R}deg_B(e) = |E(B)|.  
\end{equation}
Let $N_B(e)$ be the neighborhood of $e\in R$ in the graph $B$. 

The maximal size of $G\in \G$ is attained when the edges in  $R \setminus (T^G\cup W^G)$ are arranged in triples, so that the three edges in each triple meet at most four distinct edges of $G$.
Since $\bigcup R$ is a cover of every $G \in \G$, we have 
$$n=|G| \le 2t + \frac{3}{2}w + \frac{4}{3}(q-t-w), $$ implying 
$3n-4q \le 2t+\frac{w}{2}$. 
Summing over all $G\in \G$ we obtain
\begin{equation}\label{eqonehand}
\begin{split}
2\sum_{e \in R}|T(e)| + \frac{1}{2}\sum_{e\in R} |HW(e)| &=\sum_{G \in \G} (2|T^G|+ \frac{1}{2}|W^G|) \\& \ge (3n-4q)|\G| \ge (3n-4q)(n-q).
\end{split}
\end{equation}
\medskip

\subsection{Part (3) of Theorem \ref{main}}
Assume the conditions of Part (3) of Theorem \ref{main} hold, namely that there is no $C_3$ or $C_5$ in $\bigcup_{i=1}^n F_i$.

\begin{lemma}\label{counting}
$|T(e)|=1$, hence $\sum_{e \in R}|T(e)| \le q$.
\end{lemma}
\begin{proof}
Assume $G,H \in T(e)$. Let $g,g'\in G$ and  $h,h'\in H$ be edges intersecting $e$ and not any other edge in $R$, such that $g,h$ meet at the same vertex of $e$. Since the graph contains no triangles, the set $(R\setminus {e}) \cup \{g,h'\}$ is a rainbow matching of size $q+1$, contradicting the maximality of $R$. 
\end{proof}

\begin{lemma}\label{sameside}
Let $f\in R$ and $g_e, g_f, g_{ef} \in G$ be edges witnessing the fact that $G \in HW(e)$. Suppose $G'\in \G$, $G\neq G'$, and there exists $g'_e\in G'$ intersecting $e$ in one vertex and not intersecting any other edge in $R$.  Then $g_e \cap g'_e \cap e  \neq \emptyset$. 
\end{lemma}
\begin{proof}
We have to show that $g_e$ and $g'_e$ meet $e$ at the same vertex. If not, then  since the graph contains no triangles,  the set $(R\setminus {e}) \cup \{g_e,g'_e\}$ is a  rainbow matching of size $q+1$,  contradicting the maximality of $R$. 
\end{proof}

\begin{lemma}\label{sumdegrees}
$\sum_{e \in R}|HW(e)| \le 2q$.
\end{lemma}
\begin{proof}
 By (\ref{eqB}) it is enough to show $|E(B)|\le 2q$.

\begin{claim}\label{C4}
Let $G_1,G_2\in \G$ and $e,f\in R$. If $(e,f)$ is a half-wasteful pair for both $G_1,G_2$ then $deg_B(e)=2$.  
\end{claim}
{\em Proof of the claim.} Clearly, $deg_B(e)\ge 2$ since $G_1,G_2\in N_B(e)$. Assume to the contrary that $N_B(e)$ contains $G_3\in \G$ so that $G_3\neq G_1,G_2$. Then by Lemma \ref{sameside} the edges  in $G_1,G_2,G_3$ that meet only $e$ intersect at the same vertex of $e$. Since the graph does not contain $C_3$, this implies that the edges in $G_1,G_2$ that meat both $e$ and $f$ coincide. 

Let $a\in G_1$ be an edge meeting only $f$ in $R$, let $b \in G_2$ be an edge meeting  $e,f$ in $R$, and let $c \in G_3$ be an edge meeting only $e$ in $R$. Note that $a,c$ do not intersect since the graph contains no $C_5$. Replacing $e,f$ in $R$ with $a,b,c$  we obtain a  rainbow matching of size $q+1$, a contradiction. 
\hfill $\diamondsuit$
\medskip

Let $D =\{e\in R \mid deg_B(e) \ge 3\}$. 
\begin{claim}\label{degrees}
For every $e\in D$ there exists a subset $S(e) \subseteq R$ with the following properties:
\begin{enumerate}
    \item $|S(e)|=deg_B(e)$.
    \item For every $f\in S(e)$, $deg_B(f)=1$.  
    \item $S(e) \cap S(f) =\emptyset$ whenever  $e,f \in D$ and $e\neq f$. 
\end{enumerate}
\end{claim}
{\em Proof of the claim.} Let $e\in D$ and write $d=deg_B(e)$. Let $N_B(e)=\{G_1,\dots, G_d\}$, and let $f_1,\dots,f_d \in R$ be edges such that $(e,f_i)$ is a half-$G_i$-wasteful pair. Let $S(e)=\{f_1,\dots,f_d\}$. 

To prove (1), we have to show that $f_i\neq f_j$ if $i\neq j$. Indeed, if $f:=f_i=f_j$ then $(e,f)$ is a half-wasteful pair for both $G_i,G_j$, implying by Claim \ref{C4} that $\deg_B(e)=2$, contradicting the fact that $e\in D$.  

To prove (2), we show that $deg_B(f_i)=1$ for every $i$. First, $G_i$ is adjacent to $f_i$ in $B$, showing $deg_B(f_i)\ge 1$. Assume for contradiction that there exists $G \in \G$, $G\neq G_i$, so that $f_i$ is also a half-$G$-wasteful edge. Let $j\in [d]$ so that $G_j \neq G_i, G$ (such $j$ exists since $d\ge 3$). Let $a\in G_i$, $b\in G_j$ be edges meeting only $e$ in $R$. Let $c\in G_i$ be an edge meeting both $e,f_i$. Let $d \in G_i$ 
and $g\in G$ be edges meeting only $f$ in $R$. Then by Lemma \ref{sameside}, $a,b$ meet $e$ at the same vertex and $d,g$ meet $f$ at the same vertex. Since the graph contains no $C_3$, this implies that $b\cap c= c\cap g= \emptyset$. 
Since the graph contains no $C_5$, we have also $b\cap g= \emptyset$, implying that $(R\setminus \{e,f\})\cup \{b,c,g\}$ is a rainbow matching of size $q+1$, a contradiction.

To prove (3), let $e\neq f\in D$ and suppose $h\in S(e) \cap S(f)$.  
Let $G_1 \in HW(e,h)$ and $G_2 \in HW(f,h)$. Since $G_1$ is a matching $G_1\neq G_2$. Let $G_3 \in N_B(e)$ so that $G_3\neq G_1, G_2$ (it exists because $deg_B(e)\ge 3$). 
By Lemma \ref{sameside}, the edges in $G_1,G_2$ meeting only $h$ intersect at the same vertex of $h$.  
Let $b\in G_2$ be an edge meeting only $h$ in $R$, let $a\in G_1$ be an edge meeting $e,h$, and let $c\in G_3$ be an edge meeting only $e$ in $R$. Again, $a\cap b = a\cap c=\emptyset$ for otherwise we have a $C_3$, and $b\cap c=\emptyset$ for otherwise we have a $C_5$. Then like before, replacing $e,h$ in $R$ by $a,b,c$ results with a larger rainbow matching, a contradiction.
\hfill $\diamondsuit$
\medskip

For $i=1,2$ let $U_i=\{e\in R \mid deg_B(e) =i\}$, and let $U_3=D=R\setminus (U_1\cup U_2)$. For $i\in[3]$ let $E_i$ be the set of edges in $B$ adjacent to a vertex in $U_i$.  

By Claim \ref{degrees} $|E_3| \le |U_1|=|E_1|$. Therefore,  $$|E(B)| = \sum_{i=1}^{3} |E_i| \le 2|E_1| + |E_2| \le 2\sum_{i=1}^{3} |U_i| \le 2q.$$ This completes the proof of the lemma.
\end{proof}

Combining (\ref{eqonehand}) with   Lemmas \ref{counting} and  \ref{sumdegrees} we get 
 $$ (3n-4q)(n-q) \le 3q,$$  entailing (using similar calculations to those in Claim \ref{claimq}) $q> \frac{3n}{4}-\frac{9}{4}$. This completes the proof of Part (3) of the theorem. 

\subsection{Part (4) of Theorem \ref{main}}
Assume the conditions of Part (4) of  Theorem \ref{main} hold. That is, the matchings $F_1,\dots,F_n$ are pairwise disjoint. 

For $e\in R$, let $E_e$ be the set of edges $g\in \bigcup HW(e)$ intersecting $e$ and no other edge in $R$.

\begin{lemma}\label{sameside1}
Suppose  $e\in R$ has $deg_B(e)\ge 3$. Then all the edges in $E_e$ intersect $e$ at the same vertex (i.e., $\bigcap E_e \cap e \neq \emptyset$). 
\end{lemma}
\begin{proof}
Write $e=uv$. Let $G_1,G_2,G_3$ be three distinct matchings in $HW(e)$, and let $g_i\in G_i$ be an edge intersecting $e$ and no other edges in $R$. It is enough to show that $g_1,g_2,g_3$ meet $e$ at the same vertex. Assume $g_1,g_2$ meet $e$ at $v$, and $g_3$ meets $e$ at $u$. Since $G_1,G_2$ are disjoint, $g_1 \neq g_2$, and thus $g_3$ cannot meet both $g_1$ and $g_2$. Say $g_1\cap g_3=\emptyset$. Then  the set $(R\setminus {e}) \cup \{g_1,g_3\}$ is a  rainbow matching of size $q+1$, contradicting the maximality of $R$. 
\end{proof}

Let $D =\{e\in R \mid deg_B(e) \ge 5\}$. 
\begin{lemma}\label{degrees1}
For every $e\in D$ there exists a subset $S(e) \subseteq R$ with the following properties:
\begin{enumerate}
    \item $|S(e)|\ge \frac{deg_B(e)}{2}$.
    \item For every $f\in S(e)$, $deg_B(f)\le 2$.  
    \item If   $e,f \in D$ and $e\neq f$, then $S(e) \cap S(f) =\emptyset$. 
\end{enumerate}
\end{lemma}
\begin{proof}
Let $e\in D$ and write $d=deg_B(e)$. Let $N_B(e)=\{G_1,\dots, G_d\}$, and let $f_1,\dots,f_d \in R$ be edges so that $(e,f_i)$ is a half-$G_i$-wasteful pair. Let $S(e)=\{f_1,\dots,f_d\}$. 

To prove (1) it is enough to show that there do not exist three indices $j,k,\ell \in [d]$ such that $f_j=f_k=f_\ell$. Assume for contradiction that such $j,k,\ell$ exist, and let $f:=f_j=f_k=f_\ell$. For $i\in \{j,k,\ell\}$, let $g^i_e,g^i_f,g^i_{ef}\in G_i$  be edges witnessing the fact that $(e,f)$ is a half-$G_i$-wasteful pair. By Lemma \ref{sameside1},   $g^j_{e} \cap g^k_{e} \cap g^{\ell}_{e} \cap e \neq \emptyset$, implying $g^j_{ef} \cap g^k_{ef} \cap g^{\ell}_{ef} \cap e \neq \emptyset$. Since the edges $g^j_{ef}, g^k_{ef}, g^{\ell}_{ef}$
meet also $f$, at least two of them must be equal, contradicting the fact that the matchings are pairwise disjoint.

To prove (2), let $f\in S(e)$ and assume  $deg_B(f)\ge 3$. Let $H\in HW(e,f)$, with edges $h_e,h_f,h_{ef}\in H$ witnessing this fact. 
Let  $G,G' \in HW(f)$ be  two  matchings different than $H$. 
By Lemma \ref{sameside1}, all the edges in $E_f$ meet $f\cap h_f$, and all the edges in $E_e$ meet $e\cap h_e$. In particular, no edge in $E_f\cup E_e$ intersects $h_{ef}$. 

Let $g\in E_f \cap G$ and $g'\in E_f \cap G'$. Then   $g,g'$ are distinct. Since $\deg_B(e)\ge 4$,  $E_e \setminus (H\cup G' \cup G'')$ contains at least one edge. Thus there exist edges $a \in \{g,g'\}$ and  $b\in E_e \setminus (H\cup G' \cup G'')$ such that  $a,b$ do not intersect. Replacing $e,f$ in $R$ with $a,b$ and $h_{ef}$, we get a larger rainbow matching, a contradiction.

To prove (3), let $e\neq f\in D$ and suppose $h\in S(e) \cap S(f)$.  
Let $A \in HW(e,h)$ and $B \in HW(f,h)$. Since $A,B$ are matchings, $A\neq B$. Let $a_e,a_h,a_{eh} \in A$ and $b_f,b_h,b_{fh} \in B$ be  edges witnessing  $A \in HW(e,h)$ and $B \in HW(f,h)$, respectively.  By Lemma \ref{sameside1}, all the edges in $E_e$ meet $e$ at the same vertex, and all the edges in $E_f$ meet $f$ at the same vertex.

Split into two cases. 
\medskip

{\bf Case 1}: The edges $a_h$ and $b_h$ intersect $h$ at the same vertex. \\ Since  $|E_e \setminus \{a_e,b_e\}|\ge 2$ there exists an edge in  $c\in E_e \setminus \{a_e,b_e\}$ not intersecting $b_h$. Thus $(R\cup \{b_h, a_{eh},c\})\setminus \{e,h\}$ is a larger rainbow matching, a contradiction. 
\medskip

{\bf Case 2}: The edges $a_h$ and $b_h$ intersect $h$ at two different vertices. \\ In this case $a_{eh}, b_{fh}$ do not intersect, and clearly, no edge in $E_e\cup E_f$ intersect either $a_{eh}$ or  $b_{fh}$. Since $deg_B(e),deg_B(f)\ge 5$, there exist edges $c\in E_e \setminus \{a_e,b_e\}$ and $d\in E_f \setminus \{a_f,b_f\}$ so that $\{c,d,a_{eh}, b_{fh}\}$ is a rainbow matching in $\G$. Thus  $(R\cup \{c,d,a_{eh}, b_{fh}\})\setminus \{e,f,h\}$ is a  rainbow matching of size $q+1$, a contradiction. 
\end{proof}

\begin{lemma}\label{sumdegrees1}
$\sum_{e \in R}|HW(e)| \le 4q$
\end{lemma}
\begin{proof}
For $i\in [4]$ let $U_i$ be the set of edges in $R$ of degree $i$ in $B$, and let $U_5=D$. For $i\in[5]$, let $E_i$ be the set of edges in $B$ adjacent to a vertex in $U_i$.  
By Lemma \ref{degrees1} we have, $$|E_5| \le 2(|U_1|+|U_2|).$$ 
Therefore,  $$|E(B)| = \sum_{i=1}^{5} |E_i| \le \sum_{i=1}^{4} i|U_i| + 2(|U_1|+|U_2|) \le  4\sum_{i=1}^{5} |U_i| \le 4q.$$ This completes the proof of the lemma.
\end{proof}

Combining (\ref{eq1}), (\ref{eqonehand}) and Lemma \ref{sumdegrees1} we get 
 $$ (3n-4q)(n-q) \le  6q,$$  implying $q> \frac{3n}{4}-\frac{9}{2}$ and proving Part (4). This completes the proof of Theorem \ref{main}.

 \begin{question}
  Does there exist a function $f(r)$ such that
  $(n,n)\to_r n-f(r)$?
  \end{question}

\section{Two proofs of Theorem \ref{monotone_symmetric}}
\label{symmetricsection}

%\textcolor{blue}{Why do we need to restate?}

For a rooted tree $G$ and a vertex $v$ on it let $Gv$ be the path on $G$  from the root to $v$.  
The tree $G$ is called
{\em rainbow-monotone}  if the directed path $Gv$ is rainbow-monotone for every $v\in V(G)$.  
For a forest $\F$ whose components are rooted trees  and a vertex $v \in V(\cf)$, we write $\F v$ for the path $Gv$, where $G$ is the component of $\F$ containing $v$. Similarly, $vG$ denotes the sub-tree of $G$ rooted at $v$. 
If $K,L$ are trees sharing a vertex $v$, we denote by 
$KvL$ the edege-wise union of the path $Kv$ and the tree $vL$. This definition does not imply that $KvL$ is a tree, but in our applications it will be.

 \begin{proof}[First  proof of Theorem \ref{monotone_symmetric}]
 
 Inducting on $i$, we grow 
% collections  $\F_i, ~i \le m$ of (not necessarily disjoint) trees, where each $\F_i$ consists of  one
 trees $T_i(s)$ rooted at $s$ for every $s \in S$ and $ i \le m$.  
  At each step 
 we shall denote by $Y^k_i$ the set of vertices   $y \in Y$ that are reached by precisely $k$  trees $T_i(s)$ ($k \ge 0$).

 The inductive construction will maintain two conditions:
 \begin{enumerate}
     \item $Y_i^k =0$ for all $k>2$, and 
     
     \item $T_i(s)$ is rainbow-monotone for every $s \in S$. 
 \end{enumerate}

   %Suppose we have done so for all $j \le q$. 
For the base, $i=0$, let $T_0(s)=(s)$, a single-vertex tree, for every $s \in S$.  Condition (2) is obvious, and (1) is true since $Y^0_0=Y$ and $Y^k_0=\emptyset$ for every $k>0$. 
 
 Suppose $T_i(s)$ have
  been defined for all $s \in S$, and that they satisfy (1) and (2). 
  Let $in(P_{i+1})=p$ and $ter(P_{i+1})=q$. 
  Let $X =V(T_i(p)) \cup Y^2_i$.  Then $p \in X$. Since $q \in S$ we may assume that $q \not \in V(T_i(s))$ for any $s \neq q$, since otherwise the path $T_i(s)q$ is the desired path.  
 
 So, there is an edge $xy \in E(P_{i+1})$ such that $x \in X,  ~y \not \in X$.  
 
 {\bf Case I} $x \in V(T_i(p))$. Define then $T_{i+1}(p)=T_{i}(p) \cup \{xy\}$ (it is a tree since $y \not \in T_i(p)$) and $T_{i+1}(s)=T_{i}(s)$
 for all $s \neq p$. 
 
 {\bf Case II}
  $x \in Y_i^2\setminus V(T_i(p)$.  Since $y \not \in Y_i^2$, there exists $r\in S$ (possibly $r=p$) for which $x \in V(T_i(r))$ and $y \not \in V(T_i(r))$. Then defining  $T_{i+1}(r)=T_{i}(r) \cup \{xy\}$  and $T_{i+1}(s)=T_{i}(s)$
 for all $s \neq r$
 maintains the inductive assumptions.  
 
 Since at each of the $m$ steps we are adding a vertex in $Y$ to the trees, and since no vertex 
 in $Y$ appears more than twice, $m \le 2|Y|$, as desired. 
 \end{proof}

 The second proof is longer, but less of a hocus-pocus. It uses  part 
 (\ref{monotone_disjoint}) of Theorem 
 \ref{combinations}. We give it here since it may contain ideas relevant to Conjecture  \ref{conjpaths}, the alternating paths version of the theorem.

  \begin{proof}[Second Proof of Theorem \ref{monotone_symmetric}]

 First - a blueprint, which will then be given a rigorous formulation. 
 
% \textcolor{blue}{Can we remove the blueprint?}
 
 Assume that there is no rainbow-monotone $S-S$ path. We grow inductively rainbow-monotone  forests $\F_i$, each consisting of 
 trees $T_i(s)$ rooted at $s$, for all $s \in S$. 
 We also keep track of sets $W_i(t)$ of paths for each $ t \in S$. These are ``temporarily wasted'' paths, namely paths not adding an edge to $\F_i$. A ``wasted'' path in  $W_i(t)$ will be used at the end of the procedure to find a rainbow-monotone path ending at $t$.

 Let $\F_0$
consist of the single vertex trees $(s), ~s  \in S$, and  $W_0(t)=\emptyset$ for every  $t \in S$.

The inductive step: If there exists an edge  $xy \in E(P_{i+1})$ such that $x \in V[\F_i]$ and $y \not  \in V[\F_i]$ choose one such edge and add it to $\F_i$ to form $\F_{i+1}$. If not, let $\F_{i+1}=\F_i$ and let $W_{i+1}(ter(P_{i+1}))= W_{i}(ter(P_{i+1}))  \cup \{P_{i+1}\}$, and $W_{i+1}(t)=W_{i}(t)$ for every $t \neq ter(P_{i+1})$. 

Each non-wasted path $P_i$ adds a vertex of $Y$ to $V[\F_i]$. Hence the number of such paths is at most $|Y|$. To get the desired inequality $m \le  2|Y|$, it suffices then to show that the number of wasted paths is at most $|Y|$. This will follow from: 

$$|W_m(t)| \le |V[T_m(t)]\setminus \{t\}| ~~\text {for every}~~ t \in S.$$

To see this, assume $|W_m(t)| > |V[T_m(t)]\setminus \{t\}|$. Contract $V\setminus V[T_m(t)]$ to a single vertex $z$. 
Let   $Q_i, ~~i \le m$ consist  
of the part of the path $P_i$ contained in $V[T_m(t)]$, and let $P_i'$ be a path in the contracted graph, 
defined by $E(P_i')=E(Q_i) \cup \{z~ in(Q_i)\}$ (namely, $Q_i$ with $z$ appended to its initial vertex). 
 Then by part (\ref{monotone_disjoint})  of Theorem \ref{combinations} there is a rainbow-monotone path (with respect to the paths $P'_i$) from $z$ to $t$, which can then be extended (by uncontracting $z$) to a rainbow-monotone  $s-t$ path for some $s \neq t$. 
\medskip 

A  rigorous argument performs both types of steps together.  We use the ``wasted'' paths as we go along, instead of waiting till the end of the process.  

Below we use the symbol $\diamondsuit$ to mark the conclusion of an intermediate
step in a proof.

 We construct inductively. for $j \le m$,  non-decreasing (containment-wise) disjoint trees   $T_{{j}}(s), ~s \in S$, and non-decreasing 
 forests $\Z_j(s), ~s \in S$.

 %The union of the forests $\Z_j(s)$ will be  a pseudo-forest. 

% \textcolor{blue}{I think that the $\Z_j(s)$ are real forests}

 The inductive construction will maintain the following properties. 
 
 \begin{enumerate}[ (P1)]
 
 \item At each step precisely one $T_j(s)$ or one  $\Z_j(s)$ grows, meaning that it is added a  vertex from $Y$, not  met as yet.
 
 \item The tree $T_{j}(s)$ is rooted at $s\in S$, and $V(T_{j}(s)) \cap V(T_{j}(s'))=\emptyset$ whenever $s\neq s'$.
 
 \item The union $\T_j$ of the trees $T_{j}(s), ~s \in S,$ is a forest. 
 
 \item For $s\in S$, %\textcolor{blue}{these are forests}
  the forest $\Z_j(s)$  is the union of rooted trees, having a special form: each tree in $\Z_j(s)$ has its root $r$ in $V(G_j(s'))$ for some $s'\in S, ~s'\neq s$,   and its other vertices in $V(T_j(s))$.  We denote such a tree by  $Z_j(s,r)$.

%\item Let $R_j(s)$ be the set of roots of trees in $\Z_j(s)$. Let $\T_j(s)$ be the subgraph of   $\Z_j(s)$ induced on $V(\Z_j(s))\setminus R_j(s)$.
%Then $\T_j(s)$ and $\T_j := \bigcup_{s\in S} \T_j(s)$ are forests. 

\item  
At each step a new vertex from $Y$ is added either to  $\T_j$ or to $\Z_j(s)$ for some $s \in S$. 
 
 \item \label{zafterg}
 Edges are added to  $Z_{i}(s,r)$  for $i \ge j$ only after the path $\T_j r$ terminating at $r$ has been constructed.
 
 \item The forests $\T_j$ and $\Z_j(s)$ are rainbow-monotone.
 
 \item  If two trees, $U \in \Z_j(s_1)$ and $W \in \Z_j(s_2)$ meet, where $s_1 \neq s_2$, then they share only one vertex, which is the root of at least one of them. 
 \end{enumerate}

 (The secret wish of the trees in  $\Z_j(s)$  is to reach $s$. If fulfilled, 
the path in $\T_j r \Z_j(r,s)$ going from the the root $s'$ of the tree in $\T_j$ containing $r$ to $s$, will then be a rainbow  $S-S$ path.)

 Combined, these properties yield:
 
 \begin{claim}\label{isatree}

 The path $\T_j r$ does not meet $V(Z_j(s,r))\setminus \{r\}$, hence $\T_j r Z_j(s,r)$ is a tree.

  %The tree $Z_j(s,r)$ does not meet $V(T_j(s))$, hence $\T_j r Z_j(s,r)$
 %is a tree. \textcolor{red}{Should be: The path $\T_j r$ does not meet $V(Z_j(s,r))\setminus \{r\}$, hence $\T_j r Z_j(s,r)$ is a tree.}
 \end{claim}

 Having introduced the protagonists of the proof and their intended inductively preserved properties,  let us  describe the inductive construction. The preservation of the properties (P1)-(P8) will be mostly easy to validate - we comment only on those that are not obvious. 
 Let $\G_{(0,0)}$ consist of the singleton trees $(s), ~~s \in S$, and let
 $\Z_{(0,0)}(s)=\emptyset$ for every $s\in S$. 
 Suppose $\G_{j}$ and $\Z_{j}(s),~s\in S,$ have been defined; let $j^+=(k,i)$ be the  index following $j$ in the lexicographic order, and let  $P=P^i_k$. We shall use $P$ to extend either $\T_j$ or $\Z_j(s)$ for some $s \in S$, maintaining monotonicity. Below we assume, for contradiction, that such an extension is not possible.

 \begin{claim}\label{consistent}
 $V(P) \subseteq V(\T_j)$.
  \end{claim}
  \noindent If not, the first edge on $P$ leaving $V(\T_j)$ could be added to $\T_j$. \hfill $\diamondsuit$

\begin{claim}
The forests $\T_j$, as well as   the trees  $\T_j r Z_j(s,r)$,  are rainbow-monotone. 
 \end{claim}
 \noindent{\em Proof of the claim.}
In the construction we go over  the paths $P^k_i$ one by one. When an edge $e$ from $P^k_i$ is added, it is at the top of a tree. When we get to $P^k_j,~ j>i$, we cannot put an edge from it right after $e$, because $P^k_i$ and $P^k_j$ are vertex disjoint. So $e$ will be put on top of a path in a tree, that consists of edges from $P^t_z$ for $t<k$. This explains the monotonicity of both $\T_j(s)$ and $Z_j(s,r)$. 
By (\ref{zafterg}) it follows that their concatenation $\T_j r Z_j(s,r)$  is also rainbow-monotone.  \hfill $\diamondsuit$

%As will be seen  (``case II'' below)  the chronologically-first edge constructed in a tree $Z_j(s,r)$ is the root edge $rx$, where $x \in V(Gj(s))$, and this is done only after  the path $G_j r$ has been constructed. This guarantees the monotonicity relation between $\T_j$-edges and $\T_j$-edges in the tree $\T_j r Z_j(s,r)$. 

 %\hfill $\diamondsuit$.
 \begin{claim}\label{noz}
 If $s \in S$  then   $s \not \in V(\Z_j(s))$. 
 \end{claim}

  \noindent Indeed, if $s  \in V(Z_j(s,r))$ then the path $\T_j rZ_j(s,r)s$ terminating at $s$ is a rainbow-monotone $S-S$ path.   \hfill $\diamondsuit$
\\
 
  Let $t=ter(P)$. By Claim \ref{noz},   $t \not \in V(\Z_j(t))$. Let $v$ be the first vertex on $P$ that belongs to $V(G_j(t)) \setminus  V(\Z_j(t))$, and let $u$ be the vertex preceding it on $P$. We now add the  edge $uv$ to the forest $\Z_j(t)$. There are   two possible cases. 
  \medskip
  
  {\bf Case I.} $u  \in V(G_j(t))$. 
    Then the edge $uv$ is added to the existing tree in $\Z_j(t)$  containing $u$.
  \medskip

  {\bf Case II.} $u \not \in V(G_j(t))$. 
  It is possible that $u$ has been used already as a root of a tree $Z_j(t,u)$ in $\Z_j(t)$. 
  In this case we add the  edge $uv$ to  $Z_j(t,u)$. 
   Otherwise, a new  tree
  $Z_{j ^+} (t,u)$ is created in $\Z_{j ^+}(t)$, consisting of the single edge $uv$.

Note that in any of these cases  $\Z_j(t)$ is enlarged by the addition of the vertex $v$.  
    \medskip   
      
%This concludes the description of the construction. 

Since each step uses one path $P_i$  and since at each step either a rainbow $S-S$ path emerges, or one of the forests $\T_j$ or $\Z_j$ annexes a vertex of $Y$,   
it follows that $ m \le 2|Y|$, as desired. 
  \end{proof}

\end{document}